\documentclass[10pt]{amsart}

\usepackage{amssymb,amsmath,amsfonts,amsthm,array,comment,mathrsfs,times,graphicx}
\usepackage[latin1]{inputenc}
\usepackage{enumitem}
\usepackage{float}
\usepackage{verbatim}
\usepackage{makeidx}
\usepackage{hyperref}
\makeindex
\usepackage{geometry}
\usepackage{longtable,tabularx}
\usepackage{amsmath}
\usepackage[all,cmtip]{xy}
\usepackage{tikz-cd}
\DeclareFontFamily{U}{wncy}{}
    \DeclareFontShape{U}{wncy}{m}{n}{<->wncyr10}{}
    \DeclareSymbolFont{mcy}{U}{wncy}{m}{n}
    \DeclareMathSymbol{\Sha}{\mathord}{mcy}{"58}

\newtheorem{theorem}{Theorem}[section]
\newtheorem{corollary}[theorem]{Corollary}
\newtheorem{lemma}[theorem]{Lemma}
\newtheorem{proposition}[theorem]{Proposition}

\theoremstyle{definition}
\newtheorem{definition}[theorem]{Definition}
\newtheorem{remark}[theorem]{Remark}

\newtheorem{thmx}{Theorem}

\title{On Krull-Schmidt decompositions of unit groups of number fields}
\author{Asuka Kumon, Donghyeok Lim}
\date{}
\begin{document}

\maketitle 
\begin{abstract}
We prove that the Krull-Schmidt decomposition of the Galois module of the $p$-adic completion of algebraic units is controlled by the primes that are ramified in the Galois extension and the $S$-ideal class group. We also compute explicit upper bounds for the number of possible Galois module structures of algebraic units when the Galois group is cyclic of order $p^{2}$ or $p^{3}$.
\end{abstract}

\section{Introduction}

Let $F/K$ be a Galois extension of number fields with Galois group $G_{F/K}$. Let $\mathcal{O}^{\times}_{F}$ be the group of units of the ring of integers $\mathcal{O}_F$ of $F$ and $E_{F}:=\mathcal{O}^{\times}_F/\mu_F$ the quotient of $\mathcal{O}^{\times}_F$ by the subgroup $\mu_{F}$ of roots of unity in $F$. For every prime $p$, we write $\Sigma_{K,p}$ for the set of $p$-adic places of $K$. We write $\Sigma_{K,\infty}$ for the set of infinite places of $K$ and $R_{F/K}$ for the set of finite primes of $K$ that are ramified in $F$. For a finite set $S$ of places of $K$, we write $S_F$ for the set of places of $F$ above $S$ and $S_{F,f}$ for the subset of $S_F$ of finite primes. We also write $\mathrm{Cl}_S(F)$ for the $S$-ideal class group of $F$, which is the quotient of the class group $\mathrm{Cl}(F)$ of $F$ by the subgroup generated by the classes of primes in $S_{F,f}$.  Lastly, we write $\mathcal{E}_{F,S}$ for the pro-$p$ completion of the quotient $E_{F,S}:=\mathcal{O}^{\times}_{F,S}/\mu_F$ of the group $\mathcal{O}^{\times}_{F,S}$ of $S$-units by $\mu_F$.

Understanding the $\mathbb{Z}[G_{F/K}]$-structure of $E_F$ is challenging because it involves both the arithmetic of $F/K$ and the integral representations of $G_{F/K}$. A natural approach has been to study the $\mathbb{Z}_p[G_{F/K}]$-structure of $\mathcal{E}_F : = \mathbb{Z}_p \otimes_{\mathbb{Z}} E_{F}$ for every prime $p$ with the Krull-Schmidt theorem (cf. \cite{Bartel, Burnsunit, BLM, Duval1, Marszalek, Moser1}). If $p$ is relatively prime to $[F:K]$, then the $\mathbb{Z}_p[G_{F/K}]$-structure of $\mathcal{E}_F$ is known from the rational representation $\mathbb{Q}_p \otimes_{\mathbb{Z}_p} \mathcal{E}_F$. However, our understanding of the Krull-Schmidt decomposition of $\mathcal{E}_F$ for $p$ dividing $[F:K]$ has been limited to very special Galois groups because the classification of the $p$-adic integral representations of finite groups has proven to be notoriously difficult (cf. \cite{Duval1}). 

By the Krull-Schmidt theorem, the direct sum of all non-projective indecomposable direct summands in the Krull-Schmidt decomposition of $\mathcal{E}_F$ (resp. $\mathcal{E}_{F,S}$) is unique up to isomorphism. The isomorphism class $\mathcal{E}'_F$ (resp. $\mathcal{E}'_{F,S}$) contains much information on $\mathcal{E}_F$ (resp. $\mathcal{E}_{F,S}$) because the number of non-isomorphic projective indecomposable $\mathbb{Z}_p[G_{F/K}]$-lattices is finite and they are all cohomologically trivial. Therefore, its $\mathbb{Z}_p$-rank $c(G_{F/K}, \mathcal{E}_F)$ (resp. $c(G_{F/K}, \mathcal{E}_{F,S})$) is an important invariant by the Jordan-Zassenhaus theorem (cf. \cite[Thm. 24.2]{CurtisReiner}).

Recently, for odd primes $p$, Burns proved that $c(G_{F/K}, \mathcal{E}_{F,S})$ is bounded above by a function of $|G_{F/K}|$, $|S_{F(\zeta_p),f}|$, and the $p$-rank $\mathrm{rk}_p(\mathrm{Cl}_S(F(\zeta_p)))$ of $\mathrm{Cl}_S(F(\zeta_p))$ when $S$ contains $\Sigma_{k,p} \cup \Sigma_{k,\infty} \cup R_{F/K}$ (cf. \cite[Thm. 1.1]{Burns}). As a consequence, for a fixed finite group $G$ and a finite set $S'$ of places of $\mathbb{Q}$ containing $p$ and $\infty$, only finitely many indecomposable $\mathbb{Z}_p[G]$-lattices can appear in the Krull-Schmidt decomposition of $\mathcal{E}_{F,S'}$ in a (possibly infinite) family of $S'$-ramified Galois extensions $F/K$ of number fields with $G_{F/K} \cong G$ if $|S'_{F(\zeta_p),f}|$ and $\mathrm{rk}_p(\mathrm{Cl}_{S'}(F(\zeta_p)))$ are uniformly bounded in the family (cf. \cite[Exam. 3.2]{Burns}). This is not immediately obvious because the $\mathbb{Z}_p$-rank of $\mathcal{E}_{F,S}$ is usually unbounded in the family, whereas the work of Heller and Reiner \cite{HellerReiner2} has shown that unless the Sylow $p$-subgroup of $G$ is isomorphic to $\mathbb{Z}/p^i$ for $i \in \{0,1,2\}$, there are infinitely many non-isomorphic $\mathbb{Z}_p[G]$-lattices. This `finiteness' result was explored further in the recent works \cite{BouaLim, BLM}. In Section \ref{Burnstheoremsec}, we complement Theorem 1.1 of \cite{Burns} by proving the following theorem. In this work, we will always assume that $p$ is an odd prime.

 \begin{thmx}\label{Burnsordinary}
Let $p$ be an odd prime. Let $F/K$ be a Galois extension of number fields with $p \, | \, [F:K]$. Set $S= R_{F/K} \cup \Sigma_{K, p} \cup \Sigma_{K, \infty}$ and let $P$ be a Sylow $p$-subgroup of $G_{F/K}$. Then, we have
\begin{equation*}
    c(G_{F/K}, \mathcal{E}_F) \leq 3 \cdot |G_{F/K}| \cdot |P|^2 \cdot \big ( \, \mathrm{rk}_{p}(\mathrm{Cl}_{S}(F(\zeta_{p}))) + |S_{F(\zeta_{p}),f}| \,\, \big ).
\end{equation*}
 \end{thmx}

Theorem 1.1 of \cite{Burns} concerns $c(G_{F/K}, \mathcal{E}_{F,S})$ for $S$ containing $\Sigma_{K,p} \cup \Sigma_{K,\infty} \cup R_{F/K}$ because its proof studies the Galois cohomology group $H^1(G_{F,S}, \mathbb{Z}_p(1))$.  Although the upper bound in Theorem \ref{Burnsordinary} still depends on $S$, which needs to contain $\Sigma_{K,p}$, it implies the `existence' of arithmetic restriction on the Galois module structure of `ordinary' units $\mathcal{E}_F$.  We can also see from the proof that a similar result holds for $\mathcal{E}_{F,S}$ for any finite set $S$ of primes of $K$. Recent studies have also shown that our understanding of the Galois module structure of units has significant implications in the study of tamely ramified pro-$p$ extensions of number fields. In particular, a free $\mathbb{F}_p[G_{F/K}]$-direct summand of $O^{\times}_F/p$, if it exists, has gained considerable importance in this regard. Theorem \ref{Burnsordinary} presents a general perspective on the existence of the free components for Galois $p$-extensions (for further details, please refer to Remark \ref{Minkowski}).

In the second part of this work (Sections 3,4), we investigate the $\mathbb{Z}_p[G_{F/K}]$-structure of $\mathcal{E}_F$ for \textit{cyclic} $p$-extensions $F/K$. Although Theorem \ref{Burnsordinary} provides an arithmetic constraint for the $\mathbb{Z}_p[G_{F/K}]$-structure of $\mathcal{E}_F$, the complexity in studying $\mathcal{E}_F$ remains significant. This is because even for small $\mathbb{Z}_p$-ranks, the indecomposable $\mathbb{Z}_p[G]$-lattices have been barely understood for general finite groups $G$. Fortunately, if $G$ has a cyclic Sylow $p$-subgroup, then we can use the strong theory of Yakovlev \cite{Yakovlev1, Yakovlev2}. In that case, the $\mathbb{Z}_{p}[G]$-lattices can be classified by their Yakovlev diagrams, up to a direct sum of permutation lattices. The theory has been used to study various arithmetic objects \cite{Burns2, Burns, BurnsKumon, BLM, Castillo, vavasour2023mordell}.

Yakovlev's theory is particularly effective when dealing with (algebraic) units in cyclic $p$-extensions because the $\mathbb{Z}_p[G_{F/K}]$-structure of $\mathcal{E}_F$ is completely determined by its Yakovlev diagram (Proposition \ref{Yakovlevdetermineunit}). Hence, the theory gives us a specific object to study: the Yakovlev diagram. This saves us from having to consider the entire set of $p$-adic integral representations of $G_{F/K}$. Since the Yakovlev diagram of $\mathcal{E}_F$ consists of the first cohomology groups of $\mathcal{E}_F$, we can use genus theory to prove the following theorem. To simplify arguments, we will assume that a cyclic $p$-extension $L_n/L$ satisfies $N_{L_n/L}\mu_{L_n} = \mu_L$.

\begin{thmx}\label{2nd}
Let $L_n/L$ be a cyclic extension of number fields of degree $p^n$ for an integer $n \geq 1$ satisfying $N_{L_n/L}\mu_{L_n}=\mu_L$. Then the following claims are valid.
\begin{itemize}
\item[(i)] When $n=2$, the number of possible $\mathbb{Z}_p[G_{L_2/L}]$-module structures of $\mathcal{E}_{L_2}$ is bounded above by 
\begin{equation*}
p^{16 \cdot |R_{L_2/L}|^2+22 \cdot |R_{L_2/L}| \cdot \mathrm{rk}_p(\mathrm{Cl}_L)+9 \cdot \mathrm{rk}_p(\mathrm{Cl}_L)^2}.
\end{equation*}
\item[(ii)] When $n=3$, the number of possible $\mathbb{Z}_p[G_{L_3/L}]$-module structures of $\mathcal{E}_{L_3}$ is bounded above by 
\begin{equation*}
p^{A \cdot p \cdot |R_{L_3/L}|^{2}+B \cdot p \cdot |R_{L_3/L}| \cdot \mathrm{rk}_{p}(\mathrm{Cl}_{L}) + C \cdot p \cdot \mathrm{rk}_{p}(\mathrm{Cl}_{L})^{2}},
\end{equation*}
for some constants $A,B$ and $C$ which are independent of $L_3/L$.
\end{itemize}
\end{thmx}

It is worth noting that understanding the Krull-Schmidt decomposition of units in cyclic extensions of prime degree is much easier to study using Diederichsen's theorem \cite{Diederichsen}. We also mention that the upper bound in Theorem \ref{2nd}(i) is stronger than the one obtained by using the complete classification of indecomposable $\mathbb{Z}_p[\mathbb{Z}/p^2]$-lattices to $c(G_{L_2/L}, \mathcal{E}_{L_2})$ in Theorem \ref{Burnsordinary} when $p$ is large (cf. Remark \ref{p2}). Moreover, the upper bound in Theorem \ref{2nd}(ii) is new and significant because the $p$-adic integral representations of $\mathbb{Z}/p^i$ for $i \geq 3$ have been barely studied (cf. \cite{HellerReiner2}).

\vskip 20pt
\noindent \textbf{Notation}
\vskip 5pt
For a number field $K$ and a finite set of places $S$ of $K$, we write $K_S$ for the maximal $S$-ramified Galois extension of $K$ and $G_{K,S}$ for the Galois group of $K_S$ over $K$. For every place $v$ of $K$, we write $K_v$ for the completion of $K$ at $v$. We fix an embedding $\overline{K} \to \overline{K_v}$ from the separable closure $\overline{K}$ of $K$ into the separable closure $\overline{K_v}$ of $K_v$ and denote by $G_{K_v}$ the corresponding decomposition subgroup of the absolute Galois group $G_K$ of $K$ at a place above $v$. For a finitely generated $\mathbb{Z}_p$-module $M$, we write $\mathrm{rk}_p(M)$ for $\mathrm{dim}_{\mathbb{F}_p}(M/p)$, $\mathrm{rk}(M)$ for $\mathrm{dim}_{\mathbb{Q}_p} \mathbb{Q}_p \otimes_{\mathbb{Z}_p} M$, $M_{\mathrm{tor}}$ for the maximal torsion submodule of $M$, and $M_{\mathrm{tf}}$ for the quotient $M/M_{\mathrm{tor}}$. For a finite abelian group $\mathfrak{A}$, we write $\mathfrak{A}[m]$ for the kernel of the multiplication by $m$ map $\mathfrak{A} \xrightarrow{m} \mathfrak{A}$ for every $m \in \mathbb{N}$ and $\mathrm{v}_p(\mathfrak{A})$ for the $p$-part of its order $|\mathfrak{A}|$. Lastly, for a natural number $n \in \mathbb{N}$, we write $[n]$ for the set of integers $1 \leq i \leq n$ and $[n]^{\ast}$ for the set of integers $0 \leq i \leq n$.

\vskip 20pt

\noindent \textbf{Acknowledgements}
\vskip 5pt
We are very grateful to David Burns for suggesting this interesting project and for his encouragement and Daniel Macias Castillo for his interest in this project. We also would like to thank Dave Benson for providing references on the classification of $(\mathbb{Z}/p^2)[(\mathbb{Z}/p)]$-modules.
The second author was supported by the Core Research Institute Basic Science Research Program through the National Research Foundation of Korea (NRF) funded by the Ministry of Education (Grant No. 2019R1A6A1A11051177) and the Basic Science Research Program through the National Research Foundation of Korea (NRF) funded by the Ministry of Education (Grant No. NRF-2022R1I1A1A01071431).

\section{The Krull-Schmidt decompositions of unit groups in general Galois extensions}\label{Burnstheoremsec}

In this section, we will prove Theorem \ref{Burnsordinary}. We first review some useful facts on the $p$-adic \'{e}tale cohomology groups given in \cite{Burns}. Interested readers can refer to \cite[\S 2]{Burns} for a more comprehensive understanding and proof of these facts. As mentioned in the introduction, we assume that $p$ will always be an odd prime. We note that this condition is necessary for the validity of Proposition \ref{Galoiscohomology -1}.

Let $k$ be a fixed number field and $\Sigma$ a finite set of places of $k$ containing $\Sigma_{k,p}$ and $\Sigma_{k, \infty}$. We denote by $T$ a finitely generated free $\mathbb{Z}_p$-module equipped with a continuous $\mathbb{Z}_p$-linear action of $G_{k,\Sigma}$. Let $L/E$ be a fixed $\Sigma$-ramified Galois extension of number fields containing $k$. For every place $w \in \Sigma_L$, the Galois groups $G_{L,\Sigma}$ and $G_{L_w}$ act on $T$ via restriction maps $G_{L_w} \to G_{L,\Sigma} \to G_{k,\Sigma}$. We write $C(G_{L,\Sigma},T)$ and $C(G_{L_w}, T)$ for the standard complexes of continuous cochains.

The compactly supported \'{e}tale cohomology complex $R\Gamma_c(\mathcal{O}_{L,\Sigma}, T)$ is defined to be the mapping fibre of the natural diagonal localization maps
\begin{equation*}
C(G_{L,\Sigma}, T ) \longrightarrow \underset{w \in \Sigma_L}{\bigoplus} \, C(G_{L_w}, T)
\end{equation*}
in the derived category $D(\mathbb{Z}_p[G_{L/E}])$ of left $\mathbb{Z}_p[G_{L/E}]$-modules. We will denote its cohomology groups by $H^i_c(\mathcal{O}_{L,\Sigma}, T)$. For a finitely generated $\mathbb{Z}_p$-module $M$, we write $M^{\ast}$ and $M^{\vee}$ for $\mathrm{Hom}_{\mathbb{Z}_p}(M, \mathbb{Z}_p)$ and $\mathrm{Hom}_{\mathbb{Z}_p}(M, \mathbb{Q}_p/\mathbb{Z}_p)$ respectively. If $G_{L/E}$ acts on $M$, then we equip $M^{\ast}$ and $M^{\vee}$ with the contragradient $G_{L/E}$-action. In particular, we write $T^{\ast}(1)$ for $\mathrm{Hom}_{\mathbb{Z}_p}(T, \mathbb{Z}_p(1))=(T^{\ast})(1)$. Lastly, we set $B_L(T):= \bigoplus_{w \in \Sigma_{L,\infty}} \, H^0(G_{L_w}, T)$ and equip $B_L(T)$ with the usual semi-local action of Galois group $G_{L/E}$.

\begin{proposition}(\cite[\S 2]{Burns})\label{Galoiscohomology -1}
The following statements hold true for $L/E$, $\Sigma$, and $T$ as mentioned above.
\begin{itemize}
\item[(i)] Let $J$ be a normal subgroup of $G_{L/E}$. Then, there are natural isomorphisms in $D(\mathbb{Z}_p[G_{L^J/E}])$
\begin{equation*}
\begin{cases}
\mathbb{Z}_p[G_{L^J/E}] \otimes_{\mathbb{Z}_p[G_{L/E}]}^\mathbb{L} C(G_{L,\Sigma},T) \cong C(G_{L^J, \Sigma}, T), \\
\mathbb{Z}_p[G_{L^J/E}] \otimes_{\mathbb{Z}_p[G_{L/E}]}^\mathbb{L} R\Gamma_c(\mathcal{O}_{L,\Sigma},T) \cong R\Gamma_c(\mathcal{O}_{L^J, \Sigma}, T).
\end{cases}
\end{equation*}
\item[(ii)] The complexes $C(G_{L,\Sigma}, T)$ and $R\Gamma_c(\mathcal{O}_{L,\Sigma}, T)$ are respectively acyclic outside degrees $0,1$ and $2$ and degrees $1,2$, and $3$.
\item[(iii)] There is an exact sequence
\begin{equation}\label{prop-(iii)1}
0 \longrightarrow B_L(T) \longrightarrow H^1_c(\mathcal{O}_{L,\Sigma},T) \longrightarrow H^2(G_{L,\Sigma},T^{\ast}(1))^{\ast} \longrightarrow 0
\end{equation}
and for both $i=2$ and $i=3$ a canonical short exact sequence of $\mathbb{Z}_p[G_{L/E}]$-modules
\begin{equation}\label{prop-(iii)2}
0 \longrightarrow (H^{4-i}(G_{L,\Sigma}, T^{\ast}(1))_{\mathrm{tor}})^{\vee} \longrightarrow H^{i}_c(\mathcal{O}_{L,\Sigma}, T) \longrightarrow H^{3-i}(G_{L,\Sigma}, T^{\ast}(1))^{\ast} \longrightarrow 0.
\end{equation}

\item[(iv)] The complexes $C(G_{L,\Sigma}, T)$ and $R\Gamma_c(\mathcal{O}_{L,\Sigma}, T)$ belong to $D^p(\mathbb{Z}_p[G_{L/E}])$, the full triangulated subcategory of $D(\mathbb{Z}_p[G_{L/E}])$ comprising perfect complexes.
\item[(v)] The Euler characteristic of $R\Gamma_c(\mathcal{O}_{L,\Sigma}, T)$ in the Grothendieck group $\mathrm{K}_0(\mathbb{Z}_p[G_{L/E}])$ vanishes.
\item[(vi)] If $J$ is any subgroup of $G_{L/E}$ of odd order, then one has $\mathrm{rk}(B_L(T))= |J| \cdot \mathrm{rk}(B_{L^J}(T))$.
\end{itemize}
\end{proposition}

We also mention an algebraic result that will be useful in the proof of Theorem \ref{Burnsordinary}. For a general finite group $G$ and a $\mathbb{Z}_p[G]$-lattice $M$, let $c(G,M)$ be the sum of the $\mathbb{Z}_p$-ranks of all non-projective indecomposable direct summands in the Krull-Schmidt decomposition of $M$.

\begin{proposition}\label{algebraiclemma}
Let $G$ be a finite $p$-group and $M$ a $\mathbb{Z}_p[G]$-lattice. Then, we have 
\begin{equation*}
    c(G, M) = \mathrm{rk}(M) - |G| \cdot ( \, \mathrm{rk}_p (H^0(G, M)) - \mathrm{rk}_p (\hat{H}^0(G,M)) \, ).
\end{equation*}
\end{proposition}

\begin{proof}
According to Proposition 2.1 of \cite{Burns}, for a natural number $n$,  elements $x_1, \ldots, x_n$ of $M$ generate a $\mathbb{Z}_p[G]$-direct summand of $M$ isomorphic to $\mathbb{Z}_p[G]^n$ if and only if the classes of $\mathrm{tr}_G(x_i) := \underset{g \in G}{\sum}gx_i$ in the vector space $M^G/p$ over $\mathbb{Z}/p$ are linearly independent. As a consequence, a maximal free $\mathbb{Z}_p[G]$-direct summand of $M$ is isomorphic to $\mathbb{Z}_p[G]^m$ for $m=\mathrm{rk}_p (H^0(G, M)) - \mathrm{rk}_p (\hat{H}^0(G,M))$. By the Krull-Schmidt theorem and the indecomposability of $\mathbb{Z}_p[G]$, all projective $\mathbb{Z}_p[G]$-lattices are free. Hence, the claim follows.
\end{proof}

\begin{proof}[\textbf{Proof of Theorem \ref{Burnsordinary}}]
Let $P$ be a fixed Sylow $p$-subgroup of $G_{F/K}$. Let $F'$ be the corresponding subfield of $F$. Recall that $S$ is equal to the set $\Sigma_{K,p} \cup \Sigma_{K, \infty} \cup R_{F/K}$. We shall first prove the inequality
\begin{equation}\label{firstclaim}
c(G_{F/K}, \mathcal{E}_F) \leq [F':K] \cdot \big ( \, c(P, \mathcal{E}_{F,S}) + 2|P| \cdot |S_{F,f}| \,  \big )
\end{equation} and then compute an explicit upper bound for $c(P, \mathcal{E}_{F,S})=c(G_{F/F'}, \mathcal{E}_{F,S})$.

We can consider $\mathcal{E}_F$ also as a $\mathbb{Z}_p[P]$-lattice. Then, the $\mathbb{Z}_p[G_{F/K}]$-lattice $\mathcal{E}_F$ is isomorphic to a $\mathbb{Z}_p[G_{F/K}]$-direct summand of $\mathbb{Z}_p[G_{F/K}] \otimes_{\mathbb{Z}_p[P]} \mathcal{E}_F$ (cf. the proof of \cite[Lem. 2.6]{Burns}). This implies
\begin{equation}\label{core-Sylow}
c(G_{F/K}, \mathcal{E}_F) \leq [F':K] \cdot c(P,\mathcal{E}_F)
\end{equation}
because projective $\mathbb{Z}_p[P]$-lattices are free. Let $0 \to \mathcal{E}_F  \to \mathcal{E}_{F,S}  \to \mathfrak{C} \to 0$ be the tautological exact sequence. We have $\mathrm{rk}(\mathfrak{C}) = |S_{F,f}|$ by the Dedekind-Herbrand theorem (cf. \cite[Thm. I.3.7 in page 27]{Grasclassfieldtheory}). By Proposition \ref{algebraiclemma}, we also have
\begin{equation*}
    c(P, \mathcal{E}_F) = \mathrm{rk} (\mathcal{E}_F) - |P| \cdot \mathrm{rk}_p (H^0(P, \mathcal{E}_F)) + |P| \cdot \mathrm{rk}_p (\hat{H}^0(P, \mathcal{E}_F)).
\end{equation*}
In general, we have $\mathrm{rk}_pY \leq \mathrm{rk}_p X + \mathrm{rk}_p Z$ for every exact sequence $X \rightarrow Y \rightarrow Z$ of finitely generated $\mathbb{Z}_p$-modules. Therefore, we obtain the inequalities
\begin{equation*}
\mathrm{rk}_{p}(\hat{H}^{0}(P,\mathcal{E}_F)) \leq \mathrm{rk}_{p}(\hat{H}^{-1}(P,\mathfrak{C})) + \mathrm{rk}_{p}(\hat{H}^{0}(P,\mathcal{E}_{F,S})) \quad \text{and} \quad \mathrm{rk}_p (\mathcal{E}_{F,S}^P) \leq \mathrm{rk}_p (\mathcal{E}^P_F) + \mathrm{rk}_p(\mathfrak{C}^P)
\end{equation*}
from the long exact sequences of cohomology groups associated with the tautological exact sequence. With them, we can check that $c(P, \mathcal{E}_F)$ is bounded above by 
\begin{equation*}
\mathrm{rk}(\mathcal{E}_F) - |P| \cdot \mathrm{rk}_p(H^0(P, \mathcal{E}_{F,S})) + |P| \cdot \mathrm{rk}_p(\hat{H}^0(P, \mathcal{E}_{F,S})) + |P| \cdot \mathrm{rk}_p(\mathfrak{C}^P) + |P| \cdot \mathrm{rk}_p(\hat{H}^{-1}(P, \mathfrak{C})),
\end{equation*}
which is equal to 
\begin{equation*}
    c(P, \mathcal{E}_{F,S}) - |S_{F,f}| + |P| \cdot \mathrm{rk} (\mathfrak{C}^P) + |P| \cdot \mathrm{rk}_p(\hat{H}^{-1}(P, \mathfrak{C}))
\end{equation*}
by Proposition \ref{algebraiclemma}. The inequality $(\ref{firstclaim})$ follows from $(\ref{core-Sylow})$ and the fact that both $\mathrm{rk}_p (H^{-1}(P, \mathfrak{C}))$ and $\mathrm{rk}(\mathfrak{C}^P)$ are bounded above by $|S_{F,f}|$.
\vskip 5pt

 By Kummer theory, $\mathbb{Z}_{p} \otimes_{\mathbb{Z}} \mathcal{O}^{\times}_{F,S}$ is isomorphic to  $H^{1}(G_{F,S}, \mathbb{Z}_{p}(1))$ as $\mathbb{Z}_p[P]$-modules. By applying $(\ref{prop-(iii)2})$ for $i=2, \Sigma = S$, and $T=\mathbb{Z}_p$, we obtain the exact sequence 
\begin{equation*}
    0 \longrightarrow \big ( H^{2}(G_{F,S}, \mathbb{Z}_{p}(1))_{\mathrm{tor}} \big )^{\vee} \longrightarrow H^{2}_{c}(\mathcal{O}_{F,S}, \mathbb{Z}_{p}) \longrightarrow H^{1}(G_{F,S}, \mathbb{Z}_{p}(1))^{\ast} \longrightarrow 0.
\end{equation*}
Since $H^{2}(G_{F,S}, \mathbb{Z}_{p}(1))_{\mathrm{tor}}$ is torsion, we obtain isomorphisms of $\mathbb{Z}_{p}[P]$-modules
\begin{equation*}
H^{2}_{c}(\mathcal{O}_{F,S}, \mathbb{Z}_{p})^{\ast} \cong \big ( H^{1}(G_{F,S}, \mathbb{Z}_{p}(1))^{\ast} \big )^{\ast} \cong \big ( H^{1}(G_{F,S}, \mathbb{Z}_{p}(1))_{\mathrm{tf}} \big )^{\ast \ast} \cong (\mathcal{E}_{F,S})^{\ast \ast}.
\end{equation*}
The canonical orthonormal pairing $\mathbb{Z}_p[P] \times \mathbb{Z}_p[P] \to \mathbb{Z}_p$ induces an isomorphism $\mathbb{Z}_p[P] \cong \mathbb{Z}_p[P]^{\ast}$ of $\mathbb{Z}_p[P]$-modules. Therefore, we infer $c(P, \mathcal{E}_{F,S}) = c(P, H^2_c(\mathcal{O}_{F,S}, \mathbb{Z}_p)_{\mathrm{tf}})$. Hence, by Proposition \ref{algebraiclemma} $c(P, \mathcal{E}_{F,S})$ is equal to 
\begin{equation*}
     \mathrm{rk} (H_c^2(\mathcal{O}_{F,S}, \mathbb{Z}_p)_{\mathrm{tf}}) - |P| \cdot \mathrm{rk}_p (H^0(P, H_c^2(\mathcal{O}_{F,S}, \mathbb{Z}_p)_{\mathrm{tf}}))  + |P| \cdot \mathrm{rk}_p (\hat{H}^0(P, H^2_c(\mathcal{O}_{F,S}, \mathbb{Z}_p)_{\mathrm{tf}})).
\end{equation*}
We remark that it has been already verified in \cite[Lem. 2.7]{Burns} that 
\begin{equation}\label{lemma-Burnspaper}
\begin{aligned}
\mathrm{rk}_p(\hat{H}^0(P, H^2_c(\mathcal{O}_{F,S},\mathbb{Z}_p)_{\mathrm{tf}})) & \leq  |P| \cdot \bigg ( 2+ (1+|P|) \big ( \, \mathrm{rk}_p(\mathrm{Cl}_{S}(F(\zeta_p))) + |S_{f,F(\zeta_p)}|-1 \, \big ) \bigg ) \\
& \leq 2 \cdot |P|^2 \cdot  \big ( \, \mathrm{rk}_p (\mathrm{Cl}_S(F(\zeta_p))) + |S_{F(\zeta_p),f}| \, \big ).
\end{aligned}
\end{equation}
Thus, it remains to find an upper bound of $\mathrm{rk} ( H_c^2(\mathcal{O}_{F,S}, \mathbb{Z}_p)_{\mathrm{tf}} )  - |P| \cdot \mathrm{rk}_p ( H^0(P, H_c^2(\mathcal{O}_{F,S}, \mathbb{Z}_p)_{\mathrm{tf}}) )$.

\vskip 5pt

First, we have the equalities
\begin{equation*}
\begin{aligned}
    \mathrm{rk}(H^0(P, H^2_c(\mathcal{O}_{F,S}, \mathbb{Z}_p))) & = \mathrm{rk}(H^{2}_c(\mathcal{O}_{F',S}, \mathbb{Z}_p)) \\
    & = \mathrm{rk}(H^1_c(\mathcal{O}_{F',S}, \mathbb{Z}_p)) + \mathrm{rk}(H^3_c(\mathcal{O}_{F',S}, \mathbb{Z}_p)) \\
    & = \mathrm{rk}(B_{F'}(\mathbb{Z}_p)) + \mathrm{rk}(H^2(G_{F',S}, \mathbb{Z}_p(1))) + \mathrm{rk}(H^0(G_{F',S}, \mathbb{Z}_p(1))).
\end{aligned}
\end{equation*}
The first equality follows from Proposition \ref{Galoiscohomology -1}(i), the second one follows from Proposition \ref{Galoiscohomology -1}(v), and the last one can be checked by using the exact sequences $(\ref{prop-(iii)1})$, $(\ref{prop-(iii)2})$ for the case $L=F', i=3, \Sigma = S$, and $T=\mathbb{Z}_p$.

For any finitely generated $\mathbb{Z}_p[P]$-module $M$, it is straightforward to check that $$\mathrm{rk}_p((M_{\mathrm{tf}})^P) = \mathrm{rk}((M_{\mathrm{tf}})^P) = \mathrm{rk}(M^P).$$ As a result, we deduce 
\begin{equation}\label{equality-leftpart}
\begin{aligned}
\mathrm{rk}(H^2_c(\mathcal{O}_{F,S}, \mathbb{Z}_p)_{\mathrm{tf}}) - |P| \cdot \mathrm{rk}_p(H^0(P, H^2_c(\mathcal{O}_{F,S}, \mathbb{Z}_p)_{\mathrm{tf}})) 
& \leq \mathrm{rk}(H^2_c(\mathcal{O}_{F,S},\mathbb{Z}_p)) - |P| \cdot \mathrm{rk}(B_{F'}(\mathbb{Z}_p)) \\
& = \mathrm{rk}(H^2_c(\mathcal{O}_{F,S},\mathbb{Z}_p)) - \mathrm{rk}(B_F(\mathbb{Z}_p)) \\
& = \mathrm{rk}(H^2(G_{F,S}, \mathbb{Z}_p(1))) + \mathrm{rk}(H^0(G_{F,S}, \mathbb{Z}_p(1))) \\
& \leq \mathrm{rk}(H^2(G_{F,S}, \mathbb{Z}_p(1))) + 1.
\end{aligned}
\end{equation}
The equality on the third line follows from Lemma \ref{Galoiscohomology -1}(vi). The last equality follows from Proposition \ref{Galoiscohomology -1}(v) and the exact sequences $(\ref{prop-(iii)1})$, $(\ref{prop-(iii)2})$ for the case $L=F, i=3, \Sigma=S$, and $T=\mathbb{Z}_p$.

\vskip 5pt

By the acyclicity of $C(G_{F,S},T)$ and $C(G_{F(\zeta_p),S},T)$ in degrees larger than $2$, the first descent isomorphism in Proposition \ref{Galoiscohomology -1}(i) implies an isomorphism $$H_0(G_{F(\zeta_p)/F}, H^2(G_{F(\zeta_p),S}, \mathbb{Z}_p(1))) \cong H^2(G_{F, S}, \mathbb{Z}_p(1)).
$$ From this, we obtain
\begin{equation}\label{leftpart-ref}
\begin{aligned}
    \mathrm{rk}_p(H^2(G_{F,S}, \mathbb{Z}_p(1))) & \leq \mathrm{rk}_p(H^2(G_{F(\zeta_p),S},\mathbb{Z}_p(1))) \\
    & = \mathrm{rk}_p(H^2(G_{F(\zeta_p),S},\mathbb{Z}_p(1)/p)) \\
    & = \mathrm{rk}_p(H^2(G_{F(\zeta_p),S}, \mathbb{Z}/p)) \\
    & = \mathrm{rk}_p ( \mathrm{Cl}_S(F(\zeta_p))) + |S_{F(\zeta_p),f}| -1.
\end{aligned}
\end{equation}
The second equality follows from the acyclicity of $C(G_{F(\zeta_p),S},\mathbb{Z}_p(1))$ in degrees larger than $2$. The last equality follows from the Poitou-Tate theorem (cf. \cite[Thm. 10.7.3]{NWS}). From $(\ref{lemma-Burnspaper})$, $(\ref{equality-leftpart})$, and $(\ref{leftpart-ref})$, we infer
\begin{equation*}
    c(P, \mathcal{E}_{F,S})  \leq (2\cdot |P|^3+1) \cdot \big (\mathrm{rk}_p(\mathrm{Cl}_S(F(\zeta_p)))+|S_{F(\zeta_p),f}| \big ).
\end{equation*}
Applying the above inequality to $(\ref{firstclaim})$, we obtain
\begin{equation*}
\begin{aligned}
    c(G_{F/K}, \mathcal{E}_F) & \leq [F':K] \cdot \big ( \, (2 \cdot |P|^3+1) \cdot ( \, \mathrm{rk}_p(\mathrm{Cl}_S(F(\zeta_p))) + |S_{F(\zeta_p),f}| \, ) + 2 \cdot |P| \cdot |S_{F,f}| \, \big ) \\
    & \leq [F':K] \cdot \big ( \, (2 \cdot |P|^3+1) \cdot \mathrm{rk}_p(\mathrm{Cl}_S(F(\zeta_p))) + (2 \cdot |P|^3+2 \cdot |P|+1) \cdot |S_{F(\zeta_p),f}| \big ) \\
    & \leq 3 \cdot |G_{F/K}| \cdot |P|^2 \cdot \big ( \mathrm{rk}_p(\mathrm{Cl}_S(F(\zeta_p))) + |S_{F(\zeta_p),f}| \big ).
\end{aligned}
\end{equation*}
\end{proof}

\begin{remark}\label{Minkowski}
Theorem \ref{Burnsordinary} provides a lower bound for the $\mathbb{Z}_p$-rank of maximal projective $\mathbb{Z}_p[G_{F/K}]$-direct summand of $\mathcal{E}_F$. When $F/K$ is a $p$-extension, projective $\mathbb{Z}_p[G_{F/K}]$-lattices are free. Since the difference between $E_F$ and $\mathcal{O}_F$ is controlled by the cyclic group $\mu_F$, Theorem \ref{Burnsordinary} can be used to show the existence of free $(\mathbb{Z}/p)[G_{F/K}]$-direct summands of $\mathcal{O}^{\times}_F/p$. The free $(\mathbb{Z}/p)[G_{F/K}]$-direct summands of $\mathcal{O}^{\times}_F/p$ (cf. Minkowski units in the sense of \cite{HajirMaireRamakrishna1}) are very useful in the theory of tamely ramified pro-$p$ extensions of number fields. Let $V_{\emptyset}(F)$ be the subgroup of $F^{\times}$ consisting of elements $\alpha \in F^{\times}$ such that the principal ideal $(\alpha)$ is the $p$-th power of a fractional ideal of $F$. By the well-known exact sequence (cf. \cite[Prop. 10.7.2]{NWS})
\begin{equation*}
0 \longrightarrow \mathcal{O}^{\times}_F/p \longrightarrow V_{\emptyset}(F)/F^{\times \, p} \longrightarrow \mathrm{Cl}_{F}[p] \longrightarrow 0,
\end{equation*}
if $\mathcal{O}^{\times}_F/p$ has free components then so does $V_{\emptyset}(F)/p$ because $(\mathbb{Z}/p)[G_{F/K}]$ is a Frobenius ring. Via group-theoretic arguments and the Gras-Munnier theorem (cf. \cite[Chap. V]{Grasclassfieldtheory}, \cite{GrasMunnier}), many problems about the tamely ramified pro-$p$ extensions of number fields can be solved by using $(\mathbb{Z}/p)[G_{F/K}]$-structure of $V_{\emptyset}(F)/p$ for suitably chosen $F/K$. The universal property of the free components of $V_{\emptyset}(F)/p$ can be used to find a set $S$ of non-$p$-adic places of $F$ with desirable properties for the specific problem. Interested readers can refer to \cite{HajirMaire, HajirMaireRamakrishna1, HajirMaireRamakrishna2, LeeLim, Ozaki}.
\end{remark}

\section{The Krull-Schmidt decompositions of unit groups in cyclic $p$-extensions}\label{sectioncyclic}

From now on, we use $\Gamma$ to denote a cyclic group of order $p^n$ for a fixed natural number $n$. For each $i \in [n]^{\ast}$, we write $\Gamma_i$ for the subgroup of $\Gamma$ of order $p^i$. Let $L_n/L$ be a cyclic extension of number fields of degree $p^n$. We fix an identification $G_{L_n/L} \cong \Gamma$. To simplify the arguments, we will further assume $N_{L_n/L}(\mu_{L_n}) = \mu_L$.  This condition is equivalent to the cohomological triviality of $\mu_{L_n}$ as a $\Gamma$-module (cf. \cite[Rem. 2.2]{BLM}).

\begin{definition}
A Yakovlev diagram $(M_i, \alpha_i, \beta_i)$ is a diagram
\begin{equation*}
\begin{tikzcd}
M_{1} \arrow[r, "\alpha_{1}", shift left] & M_{2} \arrow[l, "\beta_{1}", shift left] \arrow[r, "\alpha_{2}", shift left] & \cdots \arrow[l, "\beta_{2}", shift left] \arrow[r, "\alpha_{n-2}", shift left] & M_{n-1} \arrow[l, "\beta_{n-2}", shift left] \arrow[r, "\alpha_{n-1}", shift left] & M_{n} \arrow[l, "\beta_{n-1}", shift left]
\end{tikzcd}
\end{equation*}
such that 
\begin{itemize}
\item each $M_{i}$ is a finite $(\mathbb{Z}/p^{i})[\Gamma/\Gamma_i]$-module;
\item each $\alpha_{i}$ and $\beta_{i}$ is a $\mathbb{Z}[\Gamma]$-morphism such that $\beta_i \circ \alpha_i$ and $\alpha_i \circ \beta_i$ are respectively induced by the norm $N_{\Gamma_{i+1}/\Gamma_i}: = \sum_{\gamma \in \Gamma_{i+1}/\Gamma_i} \gamma$ and the multiplication by $p$.
\end{itemize}
\end{definition}

Yakovlev diagrams form a category $\mathfrak{M}_n$ with following morphisms.

\begin{definition}
Let $(M_{i}, \alpha_{i}, \beta_{i})$ and $(M'_{i}, \alpha'_{i}, \beta'_{i})$ be two Yakovlev diagrams. A morphism from $(M_{i}, \alpha_{i}, \beta_{i})$ to $(M'_{i}, \alpha'_{i}, \beta'_{i})$ is a collection of $\mathbb{Z}[\Gamma]$-module homomorphisms $\{\kappa_{i} : M_{i} \to M'_{i}\}_{i \in [n]}$ such that $\kappa_{i+1} \circ \alpha_{i} = \alpha'_{i} \circ \kappa_{i}$ and $\beta'_{i} \circ \kappa_{i+1} = \kappa_{i} \circ \beta_{i}$ for every $i \in [n-1]$.
\end{definition}

In particular, the morphism $\{ \kappa_i : M_i \to M'_i\}_{i \in [n]}$ is an isomorphism if and only if each $\kappa_i$ is an isomorphism. Yakovlev diagrams are useful for studying the Krull-Schmidt decomposition of $\mathbb{Z}_{p}[\Gamma]$-lattices through the following propositions.

\begin{proposition}(\cite[Prop. 1.3]{Yakovlev2})
Let $M$ be a $\mathbb{Z}_{p}[\Gamma]$-lattice. Then
\begin{small}
\begin{equation*}
\begin{tikzcd}
{H^{1}(\Gamma_1,M)} \arrow[r, shift left]  & {H^{1}(\Gamma_2,M)} \arrow[l, shift left] \arrow[r,  shift left] &   \cdots \arrow[l, shift left] \arrow[r, shift left]   & {H^{1}(\Gamma_{n-1},M)} \arrow[l, shift left] \arrow[r, shift left]   & {H^{1}(\Gamma,M)} \arrow[l, shift left]
\end{tikzcd}
\end{equation*}
\end{small}
is a Yakovlev diagram, where the morphisms between cohomology groups are the restrction and corestriction maps. We shall denote this diagram by $\Delta(M)$.
\end{proposition}

\begin{proposition}\label{Yakovlevtheorem}(\cite[Thm. 2.1]{Yakovlev2})
Let $M$ and $M'$ be two $\mathbb{Z}_{p}[\Gamma]$-lattices. If $\Delta(M)$ and $\Delta(M')$ are equivalent, then there are non-negative integers $\{ m_i \}_{i \in [n]^{\ast}}$ and $\{m'_i\}_{i \in [n]^{\ast}}$ such that we have a $\mathbb{Z}_p[\Gamma]$-isomorphism
\begin{equation*}
    M \oplus \underset{i=0}{\overset{n}{\bigoplus}} \, \mathbb{Z}_p[\Gamma/\Gamma_i]^{m_i} \cong M' \oplus \underset{i=0}{\overset{n}{\bigoplus}} \, \mathbb{Z}_p[\Gamma/\Gamma_i]^{m'_i}.
\end{equation*}
\end{proposition}

For every number field $E$, we will write $I_E$ for the group of fractional ideals of $E$ and $P_E$ for the subgroup of principal fractional ideals of $E$. For every extension $E/E'$ of number fields, we will use $I_{E'}$ and $P_{E'}$ to denote also their images in $I_E$. By the assumption $N_{L_n/L}(\mu_{L_n}) = \mu_L$, we have $H^1(\Gamma_i,\mathcal{E}_{L_n}) \cong H^1(\Gamma_i, E_{L_n}) \cong H^1(\Gamma_i, \mathcal{O}^{\times}_{L_n})$ for every $i \in [n]$. Hence, $\Delta(\mathcal{E}_{L_n})$ can be studied with the following classical results (cf. \cite{Iwasawa1}). 

\begin{proposition}\label{Iwasawa}
For a cyclic extension $L_n/L$ as above the following claims are valid.
\begin{itemize}
\item[(i)] For each $i \in [n]$, we have a $\mathbb{Z}[\Gamma]$-isomorphism $H^{1}(\Gamma_i,\mathcal{E}_{L_n}) \cong P_{L_n}^{\Gamma_i}/P_{L_{n-i}}$.
\item[(ii)] Under the identification of (i), the restriction map $H^1(\Gamma_i, \mathcal{E}_{L_n}) \to H^1(\Gamma_{i-1}, \mathcal{E}_{L_n})$ corresponds to the $\mathbb{Z}[\Gamma]$-morphism  $P_{L_n}^{\Gamma_i}/P_{L_{n-i}} \to P_{L_n}^{\Gamma_{i-1}}/P_{L_{n-i+1}}$ induced by the inclusion $P_{L_n}^{\Gamma_i} \subseteq P_{L_n}^{\Gamma_{i-1}}$.
\item[(iii)] Under the identification of (i), the corestriction map $H^1(\Gamma_{i-1}, \mathcal{E}_{L_n}) \to H^1(\Gamma_i, \mathcal{E}_{L_n})$ corresponds to the $\mathbb{Z}[\Gamma]$-morphism $P_{L_n}^{\Gamma_{i-1}}/P_{L_{n-i}} \to P_{L_n}^{\Gamma_i}/P_{L_{n-i+1}}$ induced by the norm map $N_{\Gamma_i/\Gamma_{i-1}}$.
\end{itemize}
\end{proposition}

\begin{proof}
cf. \cite[Lem. 3.1]{BLM}.
\end{proof}

\begin{proposition}\label{Yakovlevdetermineunit}
Let $L_{n}/L$ be a cyclic extension of number fields as above. Then, the $\mathbb{Z}_{p}[\Gamma]$-module structure of $\mathcal{E}_{L_n}$ is determined by $\Delta(\mathcal{E}_{L_n})$ and $\mathrm{rk}(\mathcal{E}_{L_n})$.
\end{proposition}

\begin{proof}
By Nakayama's lemma, the permutation lattice $\mathbb{Z}_{p}[\Gamma/\Gamma_i]$ is indecomposable for every $i \in [n]^{\ast}$. Let $\mathcal{E}^{\dagger}_{L_n}$ be the direct sum of all indecomposable direct summands in the Krull-Schmidt decomposition of $\mathcal{E}_{L_n}$ that are not isomorphic to $\mathbb{Z}_p[\Gamma/\Gamma_i]$ for any $i \in [n]^{\ast}$. By Proposition \ref{Yakovlevtheorem}, the isomorphism class of $\mathcal{E}_{L_n}^{\dagger}$ is uniquely determined by $\Delta(\mathcal{E}_{L_n})$, and we have a $\mathbb{Z}_p[\Gamma]$-isomorphism
\begin{equation*}
    \mathcal{E}_{L_n} \cong \mathcal{E}^{\dagger}_{L_n} \oplus \bigoplus_{i \in [n]^{\ast}} \, \mathbb{Z}_p[\Gamma/\Gamma_i]^{a_i}
\end{equation*}
for some non-negative integers $\{a_i\}_{i \in [n]^{\ast}}$.
With $\mathbb{C}_p$ denoting the field of $p$-adic complex numbers, we have an isomorphism $\mathbb{C}_p \otimes_{\mathbb{Z}_p} \mathcal{E}_{L_n} \cong \mathbb{C}_p \otimes_{\mathbb{Z}_p} \mathcal{E}^{\dagger}_{L_n} \oplus \bigoplus_{i \in [n]^{\ast}}\, \mathbb{C}_p[\Gamma/\Gamma_i]^{a_i}$. By the Dirichlet-Herbrand theorem and the semi-simplicity of $\mathbb{C}_p[\Gamma]$, the $\mathbb{C}_p$-representation $\bigoplus_{i \in [n]^{\ast}}\, \mathbb{C}_p[\Gamma/\Gamma_i]^{a_i}$ is uniquely determined by $\Delta(\mathcal{E}_{L_n})$ and $\mathrm{rk}(\mathcal{E}_{L_n})$. The claim follows from the fact that for each $k \in [n]^{\ast}$, the sum $\sum_{0 \leq i \leq k}a_i$ is equal to the common multiplicity in $\bigoplus_{i \in [n]^{\ast}}\, \mathbb{C}_p[\Gamma/\Gamma_i]^{a_i}$ of $p$-adic complex characters of $\Gamma$ with kernel $\Gamma_k$.
\end{proof}

\section{The number of possible Galois structures of units in cyclic $p$-extensions}

In this section, we will study the number of possible $\mathbb{Z}_p[\Gamma]$-structures of $\mathcal{E}_{L_n}$ for a fixed cyclic extension $L_n/L$ of the previous section. When $n=1$, the structure of $\mathcal{E}_{L_1}$ is determined by the group structure of $H^1(\Gamma, \mathcal{E}_{L_1})$. Since $H^1(\Gamma, \mathcal{E}_{L_1})$ has exponent $p$, the number of possible $\mathbb{Z}_p[\Gamma]$-structures of $\mathcal{E}_{L_1}$ is bounded above by $\mathrm{rk}_p(\mathrm{Cl}_L) + |R_{L_1/L}|$ by genus theory (cf. $(\ref{genussequence2})$). As the next simplest case, we will compute explicit upper bounds for the cases $n=2,3$. According to Proposition \ref{Yakovlevdetermineunit}, this number is bounded above by the number of possible equivalence classes of $\Delta(\mathcal{E}_{L_n})$. From now on, we will use simple notations $\delta_i$ and $r_{i,j}$ to denote $\mathrm{rk}_p(\mathrm{Cl}_{L_i})$ and $|R_{L_i/L_j}|$ respectively for $0 \leq j \leq i \leq n$.

\subsection{The counting argument of Burns}

 Let
\begin{equation}\label{Yakovlevdiagram}
\begin{tikzcd}
M_{1} \arrow[r, shift left] & M_{2} \arrow[l, shift left] \arrow[r, shift left] & M_{3} \arrow[l, shift left] \arrow[r, shift left] & \cdots \arrow[l, shift left] \arrow[r, shift left] & M_{n-2} \arrow[l, shift left] \arrow[r, shift left] & M_{n-1} \arrow[l, shift left] \arrow[r, shift left] & M_{n} \arrow[l, shift left]
\end{tikzcd}
\end{equation}
be a Yakovlev diagram in $\mathfrak{M}_n$. For each finite abelian group $\mathfrak{A}$, we write $e(\mathfrak{A})$ for the exponent of $\mathfrak{A}$. In \cite[\S 3.2]{Burns2}, for fixed positive integers $d, m_1, \ldots, m_n$, Burns counted the number of possible Yakovlev diagrams $(\ref{Yakovlevdiagram})$ such that $\mathrm{rk}_p(M_i) \leq d$ and $e(M_i) \leq m_i$ for every $i \in [n]$. The counting argument was applied to Selmer groups of motives \cite{Burns2} and Mordell-Weil groups of abelian varieties \cite{Castillo}.

For a finite abelian $p$-group $\mathfrak{A}$ and $i \in [n]$, let $c_{i}(\mathfrak{A})$ be the number of conjugacy classes of the group $\mathrm{Aut}_{\mathbb{Z}_{p}}(\mathfrak{A})$ of $\mathbb{Z}_{p}$-linear automorphisms of $\mathfrak{A}$ comprising elements of order dividing $p^{n-i}$. Let $\{d_i\}_{i \in [n]}$ be a fixed set of positive integers. By the same argument as in \cite{Burns2}, we can check that the number $\mathcal{N}$ of Yakovlev diagrams $(\ref{Yakovlevdiagram})$ such that $\mathrm{rk}_{p}(M_{i}) \leq d_{i}$ for each $i \in [n]$ is bounded above by
\begin{equation}\label{counting}
\underset{J_{1} \times \cdots \times J_{n}}{\sum} \,\, \underset{i=1}{\overset{n}{\prod}}c_{i}(J_{i}) \cdot \underset{i=1}{\overset{n-1}{\prod}}\mathrm{min}\{e(J_{i}), e(J_{i+1})\}^{2\mathrm{rk}_{p}(J_{i})\mathrm{rk}_{p}(J_{i+1})},
\end{equation}
where for each $i \in [n]$, $J_{i}$ in the sum runs over all abelian $p$-groups satisfying $\mathrm{rk}_p(J_i) \leq d_i$ and $e(J_i) \leq p^{i}$.

The idea behind $(\ref{counting})$ is that the number $\mathcal{N}$ is bounded above by the product of the numbers of possible $(\mathbb{Z}/p^i)[\Gamma/\Gamma_i]$-module structures for $M_i$'s and the numbers of possible $\mathbb{Z}$-module homomorphisms between $M_i$ and $M_{i+1}$. The formula $(\ref{counting})$ can give a huge upper bound for the number of possible $\Delta(\mathcal{E}_{L_n})$ because $\mathrm{rk}_p(H^{1}(G_{i},\mathcal{E}_{L_n}))$ can be large due to the splitting of primes in $R_{L_n/L}$. In this work, we will strengthen $(\ref{counting})$ by using our specific arithmetic context more effectively.

\vskip 5pt

Our idea, presented in the following lemma, is a natural refinement of the above counting argument. For a $\mathbb{Z}_p[\Gamma]$-lattice $\mathfrak{M}$ and an integer $i \in [n]$, we will use the following notations. We remark that the Krull-Schmidt theorem works for finitely generated modules over finite unitary rings (cf. \cite[Thm. 1.4.7]{KrullSchmidt}).

\begin{itemize}
\item $\kappa_i(\mathfrak{M})$: the number of indecomposable direct summands (counted with multiplicity) in a Krull-Schmidt decomposition (as $(\mathbb{Z}/p^i)[\Gamma/\Gamma_i]$-modules) of $H^1(\Gamma_i, \mathfrak{M})$.
\item $\gamma_i(\mathfrak{M})$ : the order of a minimal system of generators of $(\mathbb{Z}/p^i)[\Gamma/\Gamma_i]$-module $H^1(\Gamma_i, \mathfrak{M}_i)$.
\item $n_i(\mathfrak{M})$ : the number of possible $(\mathbb{Z}/p^i)[\Gamma/\Gamma_i]$-structures of $H^1(\Gamma_i, \mathfrak{M})$.
\end{itemize}

\begin{lemma}\label{counting-ref}
Let $\{s_i\}_{i \in [n]}$, $\{ N_i\}_{i \in [n]}$ be fixed positive integers. Suppose that for every $i \in [n]$, we have $\kappa_i(\mathfrak{M}) \leq N_i$ and only $s_i$ designated non-isomorphic indecomposable $(\mathbb{Z}/p^i)[\Gamma/\Gamma_i]$-modules can appear in the Krull-Schmidt decomposition of $H^1(\Gamma_i, \mathfrak{M})$. Then, the following claims are valid.
\begin{itemize}
\item[(i)] $n_i(\mathfrak{M}) \leq \binom{N_i + s_i}{N_i}$.
\item[(ii)] For every $i \in [n-1]$, the number of possible restriction maps from $H^1(\Gamma_{i+1}, \mathfrak{M})$ to $H^1(\Gamma_i, \mathfrak{M})$ is bounded above by $|H^1(\Gamma_i, \mathfrak{M})^{\Gamma_{i+1}}|^{\gamma_{i+1}(\mathfrak{M})}$.
\item[(iii)] For every $i \in [n-1]$, the number of possible corestriction maps from $H^1(\Gamma_i, \mathfrak{M})$ to $H^1(\Gamma_{i+1}, \mathfrak{M})$ is bounded above by $|H^1(\Gamma_{i+1}, \mathfrak{M})[p^i]|^{\gamma_i(\mathfrak{M})}$.
\item[(iv)] The number of possible equivalence classes of $\Delta(\mathfrak{M})$ is bounded above by 
\begin{equation*}
    \prod_{i \in [n]} \, \binom{N_i + s_i}{N_i} \times \prod_{i \in [n-1]} \,  \, |H^1(\Gamma_i, \mathfrak{M})^{\Gamma_{i+1}}|^{\gamma_{i+1}(\mathfrak{M})} \times \prod_{i \in [n-1]} \, |H^1(\Gamma_{i+1}, \mathfrak{M})[p^i]|^{\gamma_i(\mathfrak{M})} .
\end{equation*}
\end{itemize}
\end{lemma}

\begin{proof}
Claim (i) follows from the Krull-Schmidt theorem and the formula for the \textit{combinations with repetitions}. The restriction and corestriction maps are determined by the images of the $\Gamma$-module generators. Claim (ii) follows from the fact $H^1(\Gamma_{i+1}, \mathfrak{M})$ is fixed by $\Gamma_{i+1}$, and claim (iii) follows from $p^iH^1(\Gamma_i, \mathfrak{M})=0$. The last claim follows from claims (i), (ii), (iii) (The reader can refer to \cite[\S 3.2]{Burns2} for more detail).
\end{proof}

\begin{remark}\label{si}
\begin{itemize}
\item[(i)] We have $s_n \leq n$ by the structure theorem for finitely generated abelian groups.
\item[(ii)] By the Jordan normal form theorem, we have $s_1 \leq p^{n-1}$.
\end{itemize}
\end{remark}

\begin{lemma}\label{nakayama}
For a $\mathbb{Z}_p[\Gamma]$-lattice $\mathfrak{M}$, we have $\gamma_i(\mathfrak{M}) \leq \mathrm{v}_p(H^1(\Gamma_i, \mathfrak{M})^{\Gamma})$ for every $i \in [n-1]$ and $\gamma_n(\mathfrak{M}) = \mathrm{rk}_p(H^1(\Gamma, \mathfrak{M}))$.
\end{lemma}

\begin{proof}
We have $\gamma_i(\mathfrak{M}) = \mathrm{rk}_p(H^1(\Gamma_i, \mathfrak{M})_{\Gamma})$ by Nakayama's lemma, where $H^1(\Gamma_i, \mathfrak{M})_{\Gamma}$ denotes the $\Gamma$-coinvariant. The first claim follows because $\mathrm{v}_p(H^1(\Gamma_i, \mathfrak{M})_{\Gamma})$ is equal to $\mathrm{v}_p(H^1(\Gamma_i, \mathfrak{M})^{\Gamma})$ by the cyclicity of $\Gamma$. The second claim is immediate.
\end{proof}

\subsection{Some consequences of genus theory}

In this subsection, we use genus theory to study $\kappa_i(\mathcal{E}_{L_n})$ and $\gamma_i(\mathcal{E}_{L_n})$. To begin with, we have the following observation about the Krull-Schmidt theorem.

\begin{lemma}\label{indecomposablesummands} Let $G$ be a finite $p$-group. Let $R$ be a finite commutative unitary ring whose order is a power of $p$. For a finitely generated $R[G]$-module $M$, let $\kappa(M)$ be the number of indecomposable direct summands (counted with multiplicity) in the Krull-Schmidt decomposition of $M$. Then, the following claims are valid.
\begin{enumerate}
\item[(i)] $\kappa(M) \leq \mathrm{v}_p(M^G)$.
\item[(ii)] If $N$ is a submodule of $M$, then one has $\kappa(N) \leq \mathrm{v}_p(M^G)$.
\item[(iii)] When $R= \mathbb{Z}/p$ and $G$ is a cyclic group, one has $\kappa(M) = \mathrm{v}_p(M^G)$.
\end{enumerate}
\end{lemma}
\begin{proof}
Let $\bigoplus_{i \in [\kappa(M)]} M_{i}$ be the Krull-Schmidt decomposition of $M$ into indecomposable $R[G]$-modules $M_{i}$'s. Then, one has $M^{G} \cong \bigoplus_{i \in [\kappa(M)]}\, M_{i}^{G}$, and claim (i) follows from $M^G_i \neq 0$ (cf. \cite[Prop. 1.6.12]{NWS}). Claim (ii) follows immediately from claim (i), and claim (iii) follows from the Jordan normal form theorem.
\end{proof}

By Proposition \ref{Iwasawa}(i), we have the exact sequence
\begin{equation}\label{genussequence2}
0 \longrightarrow (P_{L_{n}}^{\Gamma_i} \cap I_{L_{n-i}})/P_{L_{n-i}} \longrightarrow H^{1}(\Gamma_{i},\mathcal{E}_{L_{n}}) \longrightarrow P_{L_{n}}^{\Gamma_i}/(P_{L_{n}}^{\Gamma_i} \cap I_{L_{n-i}}) \longrightarrow 0,
\end{equation} 
for every $i \in [n]$. The quotient $(P_{L_{n}}^{\Gamma_i} \cap I_{L_{n-i}})/P_{L_{n-i}}$ is equal to the capitulation kernel $\mathrm{Cap}_i$ of class groups in $L_n/L_{n-i}$. For simplicity, we will use $\Pi_i$ to denote $P_{L_{n}}^{\Gamma_{i}}/(P_{L_{n}}^{\Gamma_{i}} \cap I_{L_{n-i}})$, which is a $\mathbb{Z}[\Gamma]$-submodule of $I_{L_n}^{\Gamma_i}/I_{L_{n-i}}$.

\begin{proposition}\label{fixedpartorder}
For every $i \in [n-1]$, we have
\begin{equation*}
    \mathrm{v}_p(H^{1}(\Gamma_i,\mathcal{E}_{L_n})^{\Gamma}) \leq r_{n,0} \cdot i \cdot (2+n-i) + \delta_0 \cdot i \cdot (n-i+1) - i \cdot n + i^2.
\end{equation*}
\end{proposition}

\begin{proof}
From the exact sequence ($\ref{genussequence2}$), we infer $\mathrm{v}_p(H^1(\Gamma_i, \mathcal{E}_{L_n})^{\Gamma}) \leq \mathrm{v}_p(\mathrm{Cap}^{\Gamma}_i) + \mathrm{v}_p(\Pi^{\Gamma}_i)$. The number $\mathrm{v}_p(\Pi^{\Gamma}_i)$ is bounded above by $\mathrm{v}_p((I_{L_n}^{\Gamma_i}/I_{L_{n-i}})^{\Gamma})$. For every $\mathfrak{p} \in R_{L_n/L}$, let us fix a prime $\mathfrak{p}_{L_n}$ of $L_n$ above $\mathfrak{p}$. Let $\mathfrak{p}_{L_n/L_{n-i}} \in I_{L_n}$ be the product of all the conjugates of $\mathfrak{p}_{L_n}$ over $L_{n-i}$. By a standard argument in genus theory, $I_{L_n}^{\Gamma_i}/I_{L_{n-i}}$ is generated by the classes of $\{ \mathfrak{p}_{L_n/L_{n-i}} \}_{\mathfrak{p} \in R_{L_n/L}}$ as a $\mathbb{Z}[\Gamma]$-module. Therefore, we have $\mathrm{rk}_p((I^{\Gamma_i}_{L_n}/I_{L_{n-i}})_{\Gamma}) \leq r_{n,0}$ by Nakayama's lemma. Since $I^{\Gamma_i}_{L_n}/I_{L_{n-i}}$ is annihilated by $p^i$, we deduce
\begin{equation}\label{Pi}
    \mathrm{v}_p(\Pi^{\Gamma}_i) \leq \mathrm{v}_p((I^{\Gamma_i}_{L_n}/I_{L_{n-i}})^{\Gamma}) = \mathrm{v}_p((I^{\Gamma_i}_{L_n}/I_{L_{n-i}})_{\Gamma}) \leq i \cdot r_{n,0}.
\end{equation}
On the other hand, we have
\begin{equation}\label{Cap1}
\begin{aligned}
\mathrm{v}_p(\mathrm{Cap}^{\Gamma}_i) \leq i \cdot \mathrm{rk}_p(\mathrm{Cl}^{\Gamma}_{L_{n-i}}) &  = i \cdot \mathrm{rk}_p(\mathrm{Cl}^{\Gamma/\Gamma_i}_{L_{n-i}}) \\
& \leq i \cdot \big ( \, \mathrm{rk}_p (I_{L_{n-i}}^{\Gamma/\Gamma_i}/P_{L_{n-i}}^{\Gamma/\Gamma_i}) + \mathrm{rk}_p(H^1(\Gamma/\Gamma_i, P_{L_{n-i}})) \, \big )\\
& \leq i \cdot \big ( \, \mathrm{rk}_p(I^{\Gamma/\Gamma_i}_{L_{n-i}}/P_{L}) + \mathrm{rk}_p(H^1(\Gamma/\Gamma_i, P_{L_{n-i}})) \, \big ) \\
& \leq i \cdot \big ( \, r_{n,0} + \delta_0 + \mathrm{rk}_p (H^1(\Gamma/\Gamma_i, P_{L_{n-i}})) \big ).
\end{aligned}
\end{equation}
The long exact sequence of cohomology groups associated to $1 \to \mathcal{O}_{L_{n-i}}^{\times} \to L_{n-i}^{\times} \to P_{L_{n-i}} \to 1$ yields an injective map $H^1(\Gamma/\Gamma_i, P_{L_{n-i}}) \to H^2(\Gamma/\Gamma_i, \mathcal{O}_{L_{n-i}}^{\times}) \cong \hat{H}^0(\Gamma/\Gamma_i, \mathcal{O}_{L_{n-i}}^{\times})$ by Hilbert's theorem 90. By the unit principal genus theorem (cf. \cite[Cor 2. on p. 192]{Lang}), we have
\begin{equation}\label{Cap2}
\mathrm{rk}_p(H^1(\Gamma/\Gamma_i, P_{L_{n-i}})) \leq \mathrm{v}_p(\hat{H}^0(\Gamma/\Gamma_i, \mathcal{O}_{L_{n-i}}^{\times})) = \mathrm{v}_p(H^1(\Gamma/\Gamma_i, \mathcal{O}_{L_{n-i}}^{\times})) - (n-i).
\end{equation}
By the assumption $N_{L_n/L}(\mu_{L_n}) = \mu_L$, we have the isomorphisms
\begin{equation*}
    H^1(\Gamma/\Gamma_i, \mathcal{O}_{L_{n-i}}^{\times}) \cong H^1(\Gamma/\Gamma_i, \mathbb{Z}_p \otimes_{\mathbb{Z}} \mathcal{O}_{L_{n-i}}^{\times}) \cong H^1(\Gamma/\Gamma_i, \mathcal{E}_{L_{n-i}}).
\end{equation*}
Hence, the exact sequence $(\ref{genussequence2})$ for the $\Gamma/\Gamma_i$-module $\mathcal{E}_{L_{n-i}}$ provides the inequality 
\begin{equation}\label{Cap3}
\mathrm{v}_p(H^1(\Gamma/\Gamma_i, \mathcal{O}_{L_{n-i}}^{\times})) = \mathrm{v}_p(H^1(\Gamma/\Gamma_i, \mathcal{E}_{L_{n-i}})) \leq (n-i) \cdot ( \delta_0 + r_{n,0} )
\end{equation}
The claim follows from $(\ref{Pi}), (\ref{Cap1}), (\ref{Cap2})$, and $(\ref{Cap3})$.
\end{proof}

We can derive the following corollary from Lemma \ref{nakayama}.

\begin{corollary}\label{nakayama2}
For every $i \in [n-1]$ we have
\begin{equation*}
\gamma_i(\mathcal{E}_{L_n}) \leq r_{n,0} \cdot i \cdot (2+n-i) + \delta_0 \cdot i \cdot (n-i+1) - i \cdot n + i^{2}.
\end{equation*}
\end{corollary}

\subsection{Proof of Theorem \ref{2nd}}
In order to apply Lemma \ref{counting-ref}(ii) and (iii) to $\Delta(\mathcal{E}_{L_n})$ for $n > 2$, we need the following theorem of Rosen.

\begin{proposition}(cf. \cite[Thm. 2.2]{Rosen})\label{the-Rosen}
Let $E/E'$ be a cyclic $p$-extension of number fields. Then, we have
\begin{equation*}
    \mathrm{rk}_p (\mathrm{Cl}_{E}) \leq [E:E'] \cdot \big ( \, |R_{E/E'}| + \mathrm{rk}_p(\mathrm{Cl}_{E'}) \,  \big ).
\end{equation*}
\end{proposition}

\begin{lemma}\label{res-cores-3}
Let $L_3/L$ be a cyclic extension of degree $p^3$ with $N_{L_3/L}(\mu_{L_3})=\mu_L$. The following are true.
\begin{itemize}
\item[(i)] The number of possible restriction maps $H^1(\Gamma_2, \mathcal{E}_{L_3}) \to H^1(\Gamma_1, \mathcal{E}_{L_3})$ is bounded above by 
\begin{equation*}
p^{p \cdot (5r_{3,0}+2\delta_0)(6r_{3,0}+4\delta_0)}.
\end{equation*}
\item[(ii)] The number of possible corestriction maps $H^1(\Gamma_1, \mathcal{E}_{L_3}) \to H^1(\Gamma_2, \mathcal{E}_{L_3})$ is bounded above by 
\begin{equation*}
    p^{p \cdot (2r_{3,0}+\delta_0)(4r_{3,0}+3\delta_0)}.
\end{equation*}
\end{itemize}
\end{lemma}

\begin{proof}
By Lemma \ref{counting-ref}(ii) and Corollary \ref{nakayama2}, the number of possible restriction maps is bounded above by
\begin{equation*}
    |H^1(\Gamma_1, \mathcal{E}_{L_3})^{\Gamma_2}|^{\mathrm{v}_p(H^1(\Gamma_2, \mathcal{E}_{L_3})^{\Gamma})}.
\end{equation*}
We have $\mathrm{v}_p(H^1(\Gamma_2, \mathcal{E}_{L_3})^{\Gamma}) \leq 6r_{3,0} + 4\delta_0$ by Proposition \ref{fixedpartorder} and $\mathrm{v}_p(H^1(\Gamma_1, \mathcal{E}_{L_3})^{\Gamma_2}) \leq 3r_{3,1} + 2 \delta_1 \leq p(5r_{3,0}+2\delta_0)$ by Proposition \ref{the-Rosen}. Hence, claim (i) follows.

By Lemma \ref{counting-ref}(iii) and Corollary \ref{nakayama2}, the number of possible corestriction maps is at most $$|H^1(\Gamma_2, \mathcal{E}_{L_3})[p]|^{\mathrm{v}_p(H^1(\Gamma_1, \mathcal{E}_{L_3})^{\Gamma})}= p^{\mathrm{rk}_p(H^1(\Gamma_2, \mathcal{E}_{L_3})) \cdot \mathrm{v}_p(H^1(\Gamma_1, \mathcal{E}_{L_3})^{\Gamma})}.$$
The exact sequence $(\ref{genussequence2})$ for the $\mathbb{Z}_p[\Gamma_2]$-module $\mathcal{E}_{L_3}$ provides us $\mathrm{rk}_p(H^1(\Gamma_2, \mathcal{E}_{L_3})) \leq \delta_1 + r_{3,1} \leq p \cdot (\delta_0 + 2r_{3,0})$. By Proposition \ref{fixedpartorder}, we have $\mathrm{v}_p(H^1(\Gamma_1, \mathcal{E}_{L_3})^{\Gamma}) \leq 4r_{3,0} + 3\delta_0$.
\end{proof}

The generalization of Lemma \ref{res-cores-3} to all $n \geq 4$ and $i \in [n]$ is immediate.
 Hence, the difficulty in obtaining an explicit upper bound on the number of possible $\Delta(\mathcal{E}_{L_n})$ for general $n \in \mathbb{N}$ involves $\{s_i\}_{i \in [n]}$ of Lemma \ref{counting-ref}. This requires the classification of $(\mathbb{Z}/p^i)[(\mathbb{Z}/p^j)]$-modules for various $i,j \in \mathbb{N}$. We can use Remark \ref{si} to study the case $n=2$. We remark that $(\mathbb{Z}/p^2)[(\mathbb{Z}/p)]$-modules are classified in the work of Szekeres \cite{p2-1} (also see \cite{p2-3, p2-2}), which is technically involved.

Instead, we will use the following lemma obtained with an ad-hoc method to study the case $n=3$. For every $i \in [p]$, let $Y_i$ be the indecomposable $(\mathbb{Z}/p)[\Gamma/\Gamma_2]$-module with cardinality $p^i$. We will consider $\{Y_i\}_{i \in [p]}$ also as $(\mathbb{Z}/p^2)[\Gamma/\Gamma_2]$-modules. Under a fixed identification $(\mathbb{Z}/p^2)[\Gamma/\Gamma_2] \cong (\mathbb{Z}/p^2)[X]/(X^p-1)$, $Y_i$ is isomorphic to $(\mathbb{Z}/p)[X]/((X-1)^i)$ as $(\mathbb{Z}/p^2)[X]/(X^p-1)$-modules.

\begin{lemma}\label{ad-hoc}
Let $L_3/L$ be a cyclic extension of number fields of degree $p^3$ with $N_{L_3/L}(\mu_{L_3})=\mu_L$. Then, we have 
\begin{equation*}
n_2(\mathcal{E}_{L_3}) \leq  p^{p \cdot (6r_{3,0} + 4 \delta_0)^2} \cdot \binom{6r_{3,0}+4\delta_0+p}{p}^2.
\end{equation*}
\end{lemma}

\begin{proof}
By the short exact sequence $0 \to pH^{1}(\Gamma_2, \mathcal{E}_{L_3}) \to H^1(\Gamma_2, \mathcal{E}_{L_3}) \to H^1(\Gamma_2, \mathcal{E}_{L_3})/p  \to 0$, the number of possible $(\mathbb{Z}/p^2)[\Gamma/\Gamma_2]$-module structures of $H^1(\Gamma_2, \mathcal{E}_{L_3})$ is bounded above by 
\begin{equation*}
| \, \mathrm{Ext}^1_{(\mathbb{Z}/p^2)[\Gamma/\Gamma_2]}(\,H^1(\Gamma_2, \mathcal{E}_{L_3})/p  \, , pH^1(\Gamma_2, \mathcal{E}_{L_3}) \, ) \, |
\end{equation*}
(cf. \cite[Thm. 3.4.3]{Weibel}). Since both $pH^1(\Gamma_2, \mathcal{E}_{L_3})$ and $H^1(\Gamma_2, \mathcal{E}_{L_3})/p$ have exponent $p$, they are direct sums of $(\mathbb{Z}/p^2)[\Gamma/\Gamma_2]$-modules of the form $\{Y_i\}_{i \in [p]}$. By the surjectivity of the maps
\begin{equation*}
\begin{tikzcd}
{H^1(\Gamma_2, \mathcal{E}_{L_3})} \arrow[rr, "\text{projection}"] &  & {H^1(\Gamma_2, \mathcal{E}_{L_3})/p} \arrow[rr, "p"] &  & {{p H^1(\Gamma_2, \mathcal{E}_{L_3})}}
\end{tikzcd}
\end{equation*}
and Lemma \ref{indecomposablesummands}(iii), the numbers of indecomposable direct summands (counted with multiplicity) in the Krull-Schmidt decompositions of $H^1(\Gamma_2, \mathcal{E}_{L_3})/p$ and $pH^1(\Gamma_2, \mathcal{E}_{L_3})$ are bounded above by $\mathrm{v}_p(H^1(\Gamma_2, \mathcal{E}_{L_3})^{\Gamma})$, which is bounded above by $6r_{3,0}+4\delta_0$ by Corollary \ref{nakayama2}.
By the isomorphism
\begin{equation*}
    \mathrm{Ext}^1_{(\mathbb{Z}/p^2)[\Gamma/\Gamma_2]}(\, \bigoplus_{i \in [p]} Y^{a_i}_i \, , \, \bigoplus_{j \in [p]} Y^{b_j}_j \,) \cong \bigoplus_{i,j \in [p]} \, \mathrm{Ext}^1_{(\mathbb{Z}/p^2)[\Gamma/\Gamma_2]}( \, Y_i,Y_j \, )^{a_i b_j},
\end{equation*}
 the lemma follows if we show $|\mathrm{Ext}^1_{(\mathbb{Z}/p^2)[\Gamma/\Gamma_2]}(Y_i,Y_j)| \leq p^p$ for all $i,j \in [p]$. By the base-change spectral sequence for Ext (cf. \cite[Exer. 5.6.3]{Weibel})
\begin{equation}\label{spectral}
    \mathrm{Ext}^{m}_{(\mathbb{Z}/p)[\Gamma/\Gamma_2]}(Y_{i}, \mathrm{Ext}^{n}_{(\mathbb{Z}/p^{2})[\Gamma/\Gamma_2]}(Y_{p}, Y_{j})) \Longrightarrow \mathrm{Ext}^{m+n}_{(\mathbb{Z}/p^{2})[\Gamma/\Gamma_2]}(Y_{i},Y_{j})
\end{equation}
associated with the ring homomorphism $(\mathbb{Z}/p^2)[\Gamma/\Gamma_2] \to (\mathbb{Z}/p)[\Gamma/\Gamma_2]$,  $|\mathrm{Ext}^1_{(\mathbb{Z}/p^2)[\Gamma/\Gamma_2]}(Y_i,Y_j)|$ is bounded above by the product
\begin{equation}\label{product}
      |\mathrm{Hom}_{(\mathbb{Z}/p)[\Gamma/\Gamma_2]}(Y_i, \mathrm{Ext}^1_{(\mathbb{Z}/p^2)[\Gamma/\Gamma_2]}(Y_p,Y_j))| \times |\mathrm{Ext}^1_{(\mathbb{Z}/p)[\Gamma/\Gamma_2]}(Y_i, \mathrm{Hom}_{(\mathbb{Z}/p^2)[\Gamma/\Gamma_2]}(Y_p, Y_j))|.
\end{equation}

We can check that $\mathrm{Ext}^1_{(\mathbb{Z}/p^2)[\Gamma/\Gamma_2]}(Y_p,Y_j)$ is isomorphic to $Y_j$ as $(\mathbb{Z}/p)[\Gamma/\Gamma_2]$-modules by using the projective resolution 
\begin{equation*}
\begin{tikzcd}
\cdots \arrow[r, "p"] & {(\mathbb{Z}/p^2)[\Gamma/\Gamma_2]} \arrow[rr, "p"] &  & {(\mathbb{Z}/p^2)[\Gamma/\Gamma_2]} \arrow[rr, "p"] &  & {(\mathbb{Z}/p^2)[\Gamma/\Gamma_2]} \arrow[r] & Y_p \arrow[r] & 0.
\end{tikzcd}
\end{equation*}
Therefore, the first term in $(\ref{product})$ is equal to $|\mathrm{Hom}_{(\mathbb{Z}/p)[\Gamma/\Gamma_2]}(Y_i,Y_j)| = p^{\mathrm{min}\{i,j\}}$.

On the other hand, the $(\mathbb{Z}/p)[\Gamma/\Gamma_2]$-module $\mathrm{Hom}_{(\mathbb{Z}/p^2)[\Gamma/\Gamma_2]}(Y_p, Y_j)$ arising as $\mathrm{Ext}_{(\mathbb{Z}/p^2)[\Gamma/\Gamma_2]}^0(Y_p, Y_j)$ in $(\ref{spectral})$ is isomorphic to $Y_j$. Therefore, the second term in $(\ref{product})$ is equal to $|\mathrm{Ext}^1_{(\mathbb{Z}/p)[\Gamma/\Gamma_2]}(Y_i,Y_j)|$. We can check that this is equal to $p^{\alpha(i,j)}$ for $\alpha(i,j):= \mathrm{min}\{p-i , j\} - \mathrm{max}\{j-i , 0\}$ by using the projective resolution
\begin{equation*}
    \begin{tikzcd}
\cdots \arrow[rr, "(X-1)^i"] &  & Y_p \arrow[rr, "(X-1)^{p-i}"] &  & Y_p \arrow[rr, "(X-1)^i"] &  & Y_p \arrow[rr] &  & Y_i \arrow[r] & 0.
\end{tikzcd}
\end{equation*}
The claim follows because $\alpha(i,j) + \mathrm{min}\{i,j\} \leq (p-i) - \mathrm{max}\{ j-i, 0 \} + i \leq p.$
\end{proof}

Before giving the proof of Theorem \ref{2nd}, we note that it is easy to check that
\begin{equation}\label{ine-binom}
\binom{a}{b} \leq 2^a \quad \text{for all} \, a \geq b \, \text{in} \, \mathbb{N} \quad \text{and} \quad c+d \leq d^c \quad \text{for all $c,d \in \mathbb{N}$ such that $c \geq 2$ and $d \geq 3$.}
\end{equation}

\begin{proof}[\textbf{Proof of Theorem \ref{2nd}}]
Let us first set $n=2$ and prove claim (i). We have $\mathrm{v}_p(H^1(\Gamma_1, \mathcal{E}_{L_2})^{\Gamma}) \leq 3r_{2,0} + 2\delta_0$ by Proposition \ref{fixedpartorder}, and $\mathrm{rk}_p(H^1(\Gamma, \mathcal{E}_{L_2})) \leq \delta_0 + r_{2,0}$ by the exact sequence $(\ref{genussequence2})$. By Lemma \ref{counting-ref}(iv), Remark \ref{si}, and Corollary \ref{nakayama2}, the number of possible $\Delta(\mathcal{E}_{L_2})$ is bounded above by 
\begin{equation*}
 p^{2 \cdot (r_{2,0}+\delta_0)(3r_{2,0}+2\delta_0)} \cdot \binom{3r_{2,0}+2\delta_0+p}{p} \cdot \binom{r_{2,0}+\delta_0+2}{2}.   
\end{equation*}
We can check $\binom{r_{2,0}+\delta_0+2}{2} \leq p^{r_{2,0}+\delta_0}$ by using the inequality $\delta_0 + r_{2,0} \geq 1$. Moreover, by using $(\ref{ine-binom})$ we can deduce 
\begin{equation*}
    \binom{3r_{2,0}+2\delta_0+p}{p} = \binom{3r_{2,0}+2\delta_0+p}{3r_{2,0}+2\delta_0} \leq (3r_{2,0}+2\delta_0+p)^{3r_{2,0}+2\delta_0} \leq p^{(3r_{2,0}+2\delta_0)^2},
\end{equation*}
which leads to the conclusion of the claim.
\vskip 5pt
Now, we consider the case $n=3$. By Lemma \ref{res-cores-3}, we can check that the product of the second and the third term in Lemma \ref{counting-ref}(iv) for the case $\mathfrak{M} = \mathcal{E}_{L_3}$ is bounded above by $$p^{p \cdot (\alpha' \cdot \delta_0^2 + \beta' \cdot \delta_0 \cdot r_{3,0} + \gamma' \cdot r_{3,0}^2)}$$ for some $\alpha', \beta', \gamma' > 0$ that are independent of $L_3/L$. By Remark \ref{si} and Lemma \ref{ad-hoc}, the product $n_1(\mathcal{E}_{L_3}) \cdot n_2(\mathcal{E}_{L_3}) \cdot n_3(\mathcal{E}_{L_3})$ is bounded above by
\begin{equation}\label{combination}
p^{p \cdot (6r_{3,0}+4\delta_0)^2} \cdot \binom{r_{3,0}+ \delta_0+3}{3} \cdot \binom{6r_{3,0}+4\delta_0+p}{p}^2 \cdot \binom{4r_{3,0}+3\delta_0+p^2}{p^2}.
\end{equation}
Claim (ii) follows by applying the first inequality in $(\ref{ine-binom})$ to the first two combinations in $(\ref{combination})$, and the second inequality in $(\ref{ine-binom})$ to the rightmost combination in $(\ref{combination})$ to obtain the inequality
\begin{equation*}
    \binom{4r_{3,0}+3\delta_0+p^2}{p^2} = \binom{4r_{3,0}+3\delta_0+p^2}{4r_{3,0}+3\delta_0} \leq (4r_{3,0}+3\delta_0+p^2)^{4r_{3,0}+3\delta_0} \leq p^{2 \cdot (4r_{3,0}+3\delta_0)^2}.
\end{equation*}
\end{proof}

\begin{remark}\label{p2}
We can also obtain an upper bound for the number of possible $\mathbb{Z}_p[G_{L_2/L}]$-structures of $\mathcal{E}_{L_2}$ by using the complete classification of indecomposable $\mathbb{Z}_p[(\mathbb{Z}/p^2)]$-lattices (\cite{BermanGudivok, HellerReiner1}). If we set $\Psi : = \mathrm{rk}_p(\mathrm{Cl}_S(L_2(\zeta_p))) + |S_{L_2(\zeta_p),f}|$, then $c(G_{L_2/L}, \mathcal{E}_{L_2})$ is bounded above by $3p^6 \cdot \Psi$ by Theorem \ref{Burnsordinary}. For the group $\Gamma = \mathbb{Z}/p^2$, the $4p-3$ indecomposable $\mathbb{Z}_p[\Gamma]$-lattices $M$ other than $\mathbb{Z}_p$, $\mathbb{Z}_p[\zeta_p]$, and $\mathbb{Z}_p[\Gamma/\Gamma_1]$ (cf. \cite[p. 88]{HellerReiner1}) have rational representations $\mathbb{Q}_p \otimes_{\mathbb{Z}_p} M$ that are similar to $\mathbb{Q}_p[\Gamma]$. Moreover, the $\mathbb{Z}_p[\Gamma/\Gamma_1]$-lattices $M^{\Gamma_1}$ are isomorphic to $$\mathbb{Z}_p^a \oplus \mathbb{Z}_p[\zeta_p]^b \oplus \mathbb{Z}_p[\Gamma/\Gamma_1]^c$$ for some $a,b,c \in \{ 0,1\}$. Hence, distinguishing those $4p-3$ indecomposable lattices in the Krull-Schmidt decomposition of $\mathcal{E}'_{L_2}$ is not immediate. As a result, an upper bound obtained by a naive application of the Krull-Schmidt theorem has an approximate magnitude
\begin{equation*}
    \binom{4p-3+3p^4 \cdot \Psi}{4p-3} \geq \frac{(3p^4 \cdot \Psi)^{4p-3}}{(4p)^{4p}} = \frac{p^{12p-12} \cdot (3\Psi)^{4p-3}}{4^{4p}}.
\end{equation*}
The exponent of $p$ in Theorem \ref{2nd}(i) is independent of $p$, so the upper bound of Theorem \ref{2nd}(i) is affected less by the complexity coming from integral representations of $\Gamma$, which grows with $p$.
\end{remark}

\bibliography{KumonLim}
\bibliographystyle{abbrv}

\vskip 15pt

\noindent Asuka Kumon 
\vskip 5pt
\noindent King's College London, Department of Mathematics, London WC2R 2LS, U.K \\
\textit{Email address: asuka.kumon@kcl.ac.uk}

\vskip 15pt
\noindent Donghyeok Lim 
\vskip 5pt
\noindent Institute of Mathematical Sciences, Ewha Womans University, Seoul 03760, Republic of Korea \\
\textit{Email address: donghyeokklim@gmail.com}

\end{document}